\title{Covering simplicial game complex}
\author{Neda Shojaee\footnote{\tt shojaee$_{-}$neda@aut.ac.ir}, Morteza M. Rezaii\footnote{\tt mmreza@aut.ac.ir} \\
	\small Amirkabir University of Technology, \\
	\small Polytechnic of Teheran,\\
	\small Hafez Ave, Tehran, Iran. 
}
\date{}
\documentclass[12pt]{article}
\usepackage[margin=1.5in]{geometry} 
\usepackage{epsfig,graphicx,float,pst-all}
\usepackage{amssymb}\usepackage{bbm}
\usepackage{amsmath}               
\usepackage{amsfonts}              
\usepackage{amsthm}                
\usepackage{amscd}
\usepackage{amssymb}
\usepackage{tcolorbox}
\usepackage{color}
\usepackage{algorithm,algorithmic}
\DeclareMathAlphabet{\mathpzc}{OT1}{pzc}{m}{it}

\newtcolorbox{mybox}{boxrule=0pt,width=\textwidth,arc=0mm,colback=green!20}
\newtheorem{theorem}{Theorem}[section]

\newtheorem{definition}[theorem]{Definition}
\newtheorem{example}[theorem]{Example}

\begin{document}
	
\maketitle

\begin{abstract}
In this paper we introduce a simplicial complex representation for finite non-cooperative games in the strategic form. The covering space of the simplicial game complex is introduced and we show that the covering complex is a powerful tool to find Nash Equilibrium simplices.
This representation allows us to model the cost functions of a game as a weight number on a dual vertex of the strategy situation in some stars. It yields a canonical direct sum decomposition of an arbitrary
game into three components, as the potential, harmonic 
and nonstrategic components. 
\end{abstract}

\section{Introduction}
\paragraph{}
Game theory is the study of mathematical model of strategic interaction among rational decision-makers. It has applications in all fields of social science, as well as in logic, biology, economy and computer science\cite{MDS}. Originally, it addressed zero-sum games, in which each participant's gains or losses are exactly balanced by those of the other participants. 
If our theory is more exact and strong then we can analyze more complex natural systems and obtain more optimized strategies.
The modern game theory began with the idea of mixed-strategy equilibria in two-person zero-sum games and its proof used the Brouwer fixed point theorem on continuous mappings into compact convex sets, which became a standard method in game theory and mathematical economics, \cite{JVN}. The second edition of \cite{JVN} provided an axiomatic theory of expected utility, which allowed mathematical statisticians and economists to treat decision-making under uncertainty.
Sarah Egan studied non-cooperative games as a simplicial complex, \cite{SE}. She used combination methods and used total order instead of utility functions. 
Ozan Candogan and his coworker introduced a novel flow representation for finite games in the strategic form in \cite{CMOP}. This representation allows them to develop a canonical direct sum decomposition of an arbitrary game into three components, which are referred to as the potential, harmonic and nonstrategic components. 
Chen Xilin visualized historical data of Japanese economics in 1997-2010 by the means of topological data analysis and found that periods of growth and economical crises display very different topological properties like a number of cavities on associated simplicial complexes, \cite{CX}.
In the last decades, it is tried to model games as simplicial complexes, but these methods either are so abstract or just model special cases of games. Furthermore, these methods can’t be applied in practical studies.\\
Our main goal is to find a unified method to model non-cooperative games on simplicial complexes and as a result of the different variants of the game theory model, the number of definitions of an equilibrium situation is numerous. At first, we give an algorithm for the game data to give its geometric presentation. For aiming this purpose, we make the situation simplices and define local nerve as a tool to find adjacent simplices, see Definition \ref{def:11}. Then we use the utility functions to make a weight for each simplex. Through this way, we achieve a weighted simplicial complex which is called the game complex and is denoted by $K^*_G$, see Definition \ref{Def:12}. This algorithm needs only the local data of players in the game. By using these local data, it gives us the whole structure of the game as a simplicial complex. In the next section, we equip the game simplicial complex with a discrete metric (\ref{eqn:d2}) and it allows us to define open neighborhoods of each situations by using different stars. Furthermore, we introduce a non-trivial covering space of the game complex. In every star of the covering space, we define the degree of a dual point of each situation. Then we use these degrees to find Nash Equilibrium simplex, in the simplicial game complex. In the last section, we use the dual complex of the game and the dual star operator to give some decomposition of the game complex. We see that the first class corresponds to the well-known potential games and
refer to the second class of games as harmonic games, and the third one is the non-strategic games.
 
\section{Preliminaries}

A non-cooperative game $G$ is a tuple $G= \langle N,\{S_i\}_{i\in N}, \{ u^i\}_{i\in N}\rangle$, where $n$ players numbered $1,2,...,n$ in $N$, $S_i=\{s^1_i,s^2_i,...,s^{l_i}_i\}$ is the set of all pure strategies of the $i$'th player. The set of all strategies is given by $S =\prod_{i=1}^n S_
i$ and its elements are referred to as a strategy profile. Two strategy profiles are comparable if their difference is only in the strategy of a single player. When we emphasize they differ in the $i$'th player, we use the notation $(s_i,s_{\hat{i}})$. For each $ i \in N$, $u_i$ is the $i$'th player's utility function which is defined by $ u_{i} : S \rightarrow \mathbb{R}$.
 A (pure) Nash equilibrium is a strategy profile for which no player can unilaterally deviate and improve its payoff. Formally, a pure strategy profile $ s=(s^*_i,s_{\hat{i}})$ is a Nash equilibrium if
\begin{equation*}
u_i(s^*_i,s_{\hat{i}}) \geq u_i(s_i, s_{\hat{i}}), 
\end{equation*}
for every $ {s_i} \in {S_i} $ and every $ i \in N$. In \cite{CMOP}, pure strategy profiles are considered as a vertex set $V$, and every two comparable strategy profiles are joined by a directed edge from the strategy by smaller utility function to the greater one. Through this method, each non-cooperative game is modeled as a graph with weighted edges as follows,
\begin{align*}
	w_i&:S\times S\rightarrow \mathbb{R},\\
	w_i(s,\tilde{s})&=u_i(s)-u_i(\tilde{s}),
\end{align*} 
where $s$ and $\tilde{s}$ are two $i$- comparable strategy profiles, for pairs $(s,\tilde{s})$ which are not comparable $w_i(s,\tilde{s})$ is vanished.\\
Since, players do not have a rational description to choose a pure strategy in some games, they randomly select a pure strategy. In these games, we have mixed situations which is a discrete random variable $(S,P)$ where $P=\prod_{i=1}^n P_i$, and each $P_i$ is the set of probability distributions on $S_i$. Each strategy situation in this set is denoted by $X=(x_{1},\dots,x_{n})$, and each $x_{i}$ is a probability distribution on $S_i$. In this case, the $i$'th player's payoff function is defined by, \cite{JoelWatson}
\begin{equation}\label{eqn:1}
	e_i(X)=\sum_{s\in S}u_i(s)\prod_{j\in N}x_j(s_j),
\end{equation}
where $s=(s_1,s_2,\dots,s_n)$. Comparable situations are mixed strategy situations, which are different in the one player's situation. 
For a fixed $j_i\in \{1.\dots,l_i\}$, consider the mixed strategy $x_i$ with $x_{ij_i}=1$ and $x_{ij'_i}=0$ for all $j'_i\in \{1,\dots,l_i\}$ with $j_i\neq j'_i$. This is equivalent to the $i$'th player strategy profile $s^j_i\in S_i$ and this is true for all $i$'th player's strategies, so we have $S_i\subset P_i$.
\begin{definition}\label{Def:1}
	The mixed situation $X=(x_1,\dots,x_n)$ is a mixed Nash Equilibrium if for all $i\in N$ and every $s_i\in S_i$ we have
	\begin{equation}\label{eqn:2}
		e_i(x_i,X_{\hat{i}})\geq e_i(s_i,X_{\hat{i}}).
	\end{equation} 
\end{definition}
Let $K$ be a simplicial complex with vertex set $\{v_0,v_1,\dots,v_n\}$, denote the vector space generated by a basis consisting of oriented $p$-simplices $[v_0,\dots,v_p]$ by $C_p (K;\mathbb{R})$. This is the space of finite formal sums of the oriented $p$-simplices, with coefficients in $\mathbb{R}$. Elements of $C_p(K;\mathbb{R})$ are called $p$-chains. The discrete $p$-forms are maps from the space of $p$-chains to $\mathbb{R}$.
\begin{definition}\cite{ANH}
	A discrete $p$-form $w$ is a homomorphism from the chain space $C_p (K;\mathbb{R})$ to the
	vector space $\mathbb{R}$. Thus a discrete $p$-form is an element of $C^p(K,\mathbb{R}):=Hom(C_p (K),\mathbb{R})$, the space of cochains.
\end{definition}
The boundary operator is a homomorphism defined by
\begin{align*}
	\partial_{p} &: C_p(K;\mathbb{R}) \rightarrow C_{p-1}(K;\mathbb{R}),\\
	\partial_{p} ( \sigma_p) &= \sum_{i=0}^{p} (-1)^{i}[v^{0} , ... , v^{i-1} , v^{i+1} , \dots , v^{p}] .
\end{align*}
This map enables us to define the chain complex as follow
\begin{equation}\label{eqn:3}
	0\rightarrow
	C_{k} \xrightarrow{ \partial_{k-1}} C_{k-1} \xrightarrow{ \partial_{k-2}}\dots \xrightarrow{ \partial_{1}} C_{1}\xrightarrow{ \partial_{0}} C_{0}\rightarrow 0.
\end{equation}
The boundary operator $\partial_{t}$ satisfied $\partial_{t} \circ \partial_{t-1} = 0 $. So $ im \, \partial_{t-1}$ is contained in $ker \, \partial_{t}$ for every $0\leq t \leq n-1 $, and the quotient 
$ ker \, \partial_{t} /  im \, \partial_{t-1} $ is well defined. 
We call this quotient the $t$'th homology group of $K$.
\begin{equation}\nonumber
	H_{t} =\frac{ker \, \partial_{t}} {im \, \partial_{t-1}}. 
\end{equation}
Hence, $H(K) = \bigoplus_{t=0}^{n-1}H_{t}(K)$ is a graded Abelian group which we call it the homology of $K$.
The co-boundary operator is a homomorphism $\partial^*_{p} : C^p(K;\mathbb{R}) \rightarrow C^{p+1}(K;\mathbb{R}),$ defined by
duality to the boundary
operator using the natural bi-linear pairing between discrete forms and chains. Specifically, for a discrete form $w\in C^P(K,\mathbb{R})$ and a chain $\sigma_{p+1}\in C_{p+1}(K,\mathbb{R})$, define $\partial^*_p$ by 
\begin{align*}
	(\partial^*_pw_p,\sigma_{p+1})=(w_p,\partial_{p+1}\sigma_{p+1}).
\end{align*}
This map enables us to define the cochain complex as follow\cite{ANH}
\begin{equation}\label{eqn:0}
	0\leftarrow
	 C^{k} \xleftarrow{ \partial^*_{k-1}} C^{k-1} \xleftarrow{ \partial^*_{k-2}} \dots \xleftarrow{ \partial^*_{1}} C^{1}\xleftarrow{ \partial^*_{0}} C^{0}\leftarrow 0.
\end{equation}	 
For the co-boundary operator, we have $\partial^*_{t-1} \circ\partial^*_{t} =0$, as well. 

\section{Simplicial game Complex} 
Let $G$ be a non-cooperative $n$ player game. The space of situations on $S_i$ is an infinite dimensional vector space. In the practical studies, the space is the finite dimensional subspace according to the assumptions on a problem.
Furthermore, suppose 
$\{x_{i_k}\}_{k=1}^{m_i}$ 
is the set of possible probability distributions on a strategy profile $S_i$. Hence, we have for each $x_{i_k}$ 
\begin{align*}
	x_{i_k}:S_i\rightarrow [0,1],\\
	x_{i_k}(s_{i}^{j_i})=x_{i_kj_i},
\end{align*}
where $\sum_{j_i=1}^{l_i}x_{i_k}(s_{i}^{j_i})=1.$
Consider the set of all strategy situations $P=\prod_{i=1}^nP_i$, where $P_i=span\{x_{i_k}\}_{k=1}^{m_i}$. 
We suppose that the mixed situations $X_{\alpha}$ in $P$ are facets of a simplicial complex which is denoted by $K_G$.
Furthermore, $C_{n-1}(K_G)$ is a finite dimensional vector space with oriented simplices $\{X_{\alpha}\}_{\alpha=1}^{\prod_im_i}$ as its basis, where $X_{\alpha}=[x_{\alpha_1},\dots,x_{\alpha_n}]$.
Fix $1\leq i \leq n$, let 
$X_{\hat{i}}$ 
be a strategy situation which is obtained by omitting the mixed strategy of the $i$'th player, i.e.
$$X_{\hat{i}}=(x_1,\dots,\hat{x_i},\dots,x_n),$$ so $X_{\hat{i}}\in P_{\hat{i}}:=P_1\times P_2\times \dots P_{i-1}\times P_{i+1}\dots \times P_n$. 
For every $0\leq t\leq n-2$, the collection of $t$-dimensional faces of mixed situations is a vector space which is denoted by $C_t(K_G)$. Using the boundary operator allows us to have the chain complex  (\ref{eqn:3}) for $K_G$.
The barycentric point of mixed situations $\sigma_{n-1}=[x_1,\dots,x_{n}]$  
is defined by
\begin{equation}\label{eqn:07}
	bc(\sigma_{n-1})=\sum_{s\in S}\frac{\sum_{j=1}^{n}x_{j}(s^{k_j}_j)x_j}{\sum_{j=1}^n x_{jk_j}}.
\end{equation}
Through extending every pure face $\sigma_t$ $(0\leq t\leq n-2)$, of mixed situations to a $(n-1)$-simplex by adding zero function in suitable situations, we can use equation (\ref{eqn:07}) and find the barycentric point of each pure face and using equation (\ref{eqn:1}) to equip $\sigma_t$ with some weight for every $t$, $0\leq t\leq n-1$. 
Obviously, the barycentric point $bc({\sigma})$ of each simplex $\sigma$ in $K_G$ is inside $\sigma$. Use these barycentric points to make the barycentric subdivision of $K_G$ and denote it by $bcsd(K_G)$.
\begin{definition}\label{Def:12}
The weighted complex ${K}^*_G$ with facets ${\sigma}_{n-1}$ of $K_G$ and all pure simplices which admit weight numbers by
 \begin{equation}\label{eqn:08}
  f_{t}(\sigma_t)=\sum_{i=1}^ne_i(\sigma_t),
 \end{equation}
is called the game complex of the game $G$. So each simplex in $K^*_G$ is in the form $\sigma^*_{t}=(\sigma_t,f_t)$ where $f_{t}:C_t(K_G)\rightarrow \mathbb{R}$.
\end{definition}
An $i$-comparable situation of a game complex $K^*_G$ is a star 
\begin{equation}\label{eqn:4}
	st(X_{\hat{i}})=\bigcup_{x\in P_i}x\vee X_{\hat{i}}=\bigcup_{i_k=1}^{m_i}x_{i_k}\vee X_{\hat{i}},
\end{equation}
for which each situation equip with a weight number $e_i(\sigma)$. So for every $i\in \{1,\dots,n\}$ and each $X_{\alpha_{\hat{i}}}$, the star $st(X_{\alpha_{\hat{i}}})$ is a collection of $i$-comparable situations.
The barycentric point of each simplex in a comparable situation, $st(X_{\hat{i}})$ is
 \begin{equation}\label{eqn:5}
 	bc({x_{i_k}\vee X_{\hat{i}}})=\sum_{(s_i,s_{\hat{i}})\in S}\frac{\sum_{j=1,j\neq i}^{n}x_{j}(s^{k_j}_j)x_j+x_{i_k}(s_i)x_{ik}}{\sum_{j=1,j\neq i}^n x_{jk_j}+x_{i_ki}}.
 \end{equation}
 For convenience, we denote the strategy situation $x_{i_k}\vee X_{\hat{i}}$ by $\sigma_{i_k}$ and its barycentric by $bc({\sigma}_{i_k}).$ 
 Furthermore, we can use the barycentric points to determine components which are used to make the $(n-t-1)$-dimensional dual cell $D(\sigma_t)$ associated to a $t$-simplex ${\sigma}_t$.
 For aiming this purpose, consider the star duality operator as follows
 \begin{align}\label{eqn05}
*:&C_t(K_G,\mathbb{R})\rightarrow C_{n-t-1}(bcsd(K_G),\mathbb{R}),\\\nonumber
*(\sigma_t)=&\sum_{\sigma_t<\dots<\sigma_{n-1}}sign(\sigma_t\dots\sigma_{n-1})[bc({\sigma_t}),\dots,bc({\sigma}_{n-1})],
 \end{align}
where the sign coefficient is chosen as $\pm 1$ for all $0\leq t\leq n-1$ according to the induced orientation, \cite{ANH}. We note that as sets the equation (\ref{eqn05}) form a cell complex identical to $D(K_G)$ which is the dual complex of $K_G$. It is easy to see that $D(\sigma_{n-1})=*{\sigma}_{n-1}=bc(\sigma_{n-1})$ so we call the barycentric points of situation simplices the dual points. If $\sigma_{n-2}$ is a boundary of $K_G$ then we can write  $$D(\sigma_{n-2})=*\sigma_{n-2}=sign(bc(\sigma_{n-2}),bc(\sigma_{n-1}))[bc(\sigma_{n-2}), bc(\sigma_{n-1})],$$
as sets equality since, $$D(\sigma_{n-2})=\bigcup_{\sigma_{n-2}<\sigma_{n-1}\\ \text{and}\\ \sigma_{n-2}<\beta_{n-1}}[bc(\sigma_{n-1}),bc(\sigma_{n-1})],$$
for which its direction is induced by orientation of $\sigma_{n-1}$ and $\beta_{n-1}$. Now, we equip each dual edge by a weight number, as follows. 
Let $\hat{V}$ be the set of its dual vertices, the weighted dual flow is
 \begin{align} \label{eqn:7}
 	\omega &: \hat{V}\times\hat{V} \rightarrow \mathbb{R},\\\nonumber
 	\omega(\hat{\sigma},\hat{\beta}) &=
 	\begin{cases}
 		e_i(\sigma)-e_i(\beta), & \text{if $\sigma$ and $\beta\in st(X_{\hat{i}})$ for some $i$}\\
 		0, & o.w
 	\end{cases},
 \end{align}
 which is a directed edge between dual points with the weight number $\omega(\hat{\sigma},\hat{\beta})$ and direction from the node by smaller weight to the node with bigger one. 
 \begin{definition}\label{def:11}
 Let $K^*_G$ be a simplicial game complex. The spanning tree which is made by the union of dual flows of each comparable star is called the local nerve of $K^*_G$. The gluing of all spanning trees of all comparable stars in $K^*_G$ by a label function $l$ is called the global nerve of $K^*_G$. 
 \end{definition}
It is obvious that there is an edge between two dual vertices if associated simplices have the same adjacent $(n-2)$-simplex.  
For every $i$, we have $|P_{\hat{i}}|$ disjoint comparable stars and consequently local nerves. These local nerves are used as a tool to make the nerve of a whole complex which is called the simplicial game complex. Let $l:\{X_{\alpha}\}\rightarrow m$ be a labeling function.
So each label distinguishes a unique base situation of the space $C_{n-1}(K_G)$. All disjoint nerves are Labeled by $l$. These local nerves connect by gluing to each other through the dual vertices of the same labels. The resulting global graph is the global nerve of a simplicial complex $K^*_G$, cf., Algorithm \ref{algo}. 
\begin{algorithm}
		\caption{Simplicial Game Complex, $K^*_G$}
		\begin{algorithmic}\label{algo}
	\REQUIRE{The situation set $P$, the number of players $n$, M=\{\}, ST=\{\}.}
	\ENSURE{A simplicial game complex, $K^*_G$.}
	\FORALL{strategy situation $X_{\alpha}\in P$}
		\FORALL{$i\in \{1,\dots, n\}$} 
			\STATE \text{construct the star $st({X_{\alpha}}_{\hat{i}})$\;}
			\STATE \text{construct the set $ST=ST\cup\{{st(X_{\alpha}}_{\hat{i}})\}$.} 
			\ENDFOR
		\STATE \text{construct the label function $l:P\rightarrow \{1,\dots,\prod_{i}m_i\}$.}
		\ENDFOR
		\FORALL{$st({X_{\alpha}}_{\hat{i}}) \in ST$} 
		\REPEAT   
		\FORALL{strategy situation $\sigma_j\in ST$} 
		\STATE \text{ compute the dual point $\hat{\sigma_j}$ and associated weight of $\sigma_j$} \STATE \text{ using equations
			(\ref{eqn:5}) and (\ref{eqn:08}), respectively.}
		\ENDFOR
		\STATE \text{put all dual points to a queue, randomly.}
		\STATE {${Nl_{\alpha}}_{\hat{i}}=\{\}$.}
		\WHILE {the queue is not empty} 
		\STATE \text{pop the head dual point $\hat{\sigma_k}$ from the queue.}
		\STATE \text{construct the dual edge $E_k=D(X_{\alpha_k}).$}
		\STATE{${Nl_{\alpha}}_{\hat{i}} = {Nl_{\alpha}}_{\hat{i}}\cup\{E_k\}$.}
		\ENDWHILE
		\STATE{$M = M\cup st({X_{\alpha}}_{\hat{i}}).$}
		\UNTIL{$ST-M=\{\}$}
		\ENDFOR
		\FORALL {$i\in\{1,\dots,\prod_{i=1}^n{m_i}\}$}
			\FORALL {$Nl_{\alpha_{\hat{i}}}\cap Nl_{\beta_{\hat{j}}} $}
				\STATE \text{conspiratorially glue the nerve ${Nl_{\alpha}}_{\hat{i}}$ and ${Nl_{\beta}}_{\hat{j}}$ by $f$,}
				\STATE\text{ for each 
					vertex $\hat{\sigma}$ and $\hat{\gamma}$ in $Nl_{\alpha_{\hat{i}}}\cap Nl_{\beta_{\hat{j}}}$ with $l(\sigma)=l(\gamma)$.}
				\ENDFOR
			\ENDFOR
		\STATE	{randomly choose the vertex of a global nerve, $\hat{\sigma_0}$.}
		\STATE	\text{embed its associated strategy situation, $\sigma_0$ and label it by (\ref{eqn:08}).} 
		\STATE	\text{put all the neighboring vertices of the random leaf to a queue.}  
		\WHILE{the queue is not empty} 
				\STATE \text{pop the head vertex $\hat{\sigma_k}$ from the queue.}
				\STATE	\text{embed its associated simplex $\sigma_k$ and label it by (\ref{eqn:08}).}
				\STATE \text{put the neighboring vertices of $\hat{\sigma_k}$,}
				\STATE\text{which have not been accessed to the queue, yet.}
		\ENDWHILE
\end{algorithmic}
\end{algorithm}
\begin{example}
	Let $G$ be the paper-scissor-crop game. Suppose player $1$ has three mixed strategies, $x_1,x_2,x_3$ and player $2$ has three mixed strategies, $y_1,y_2,y_3$ which are probability distributions on $S_1=\{P,S,C\}=S_2$. The situation simplices are $1$-dimensional oriented edges $[x_i,y_j]$ for $1\leq i,j\leq 3$. It is obvious that there exist $9$ different simplices in the mentioned form which make $P$. Use label function which is different colors in our vertices here as Figure \ref{ln1}. For every $1\leq i\leq 2$ and each edge in $P$, we make $st(X_{\alpha_{\hat{i}}})$ by equation (\ref{eqn:4}). Then make the dual vertex and its weight of situation simplex in each comparable star by equations (\ref{eqn:5}) and (\ref{eqn:08}), respectively. Using these data to make dual flows of each comparable stars and its local nerve. Finally, we make the global nerve by connecting these six local nerves through the same color vertices, see Figure \ref{nerve}. By using the method of Algorithm \ref{algo}, the complex $K_G$ will be Figure \ref{SGC} which is a $1$-dimensional simplicial complex in $\mathbb{R}^3$. Actually, the simplicial game complex is the weighted boundary of Mobius strip.
	\begin{figure}[ht!]
		\begin{center}
			\includegraphics[scale=0.45]{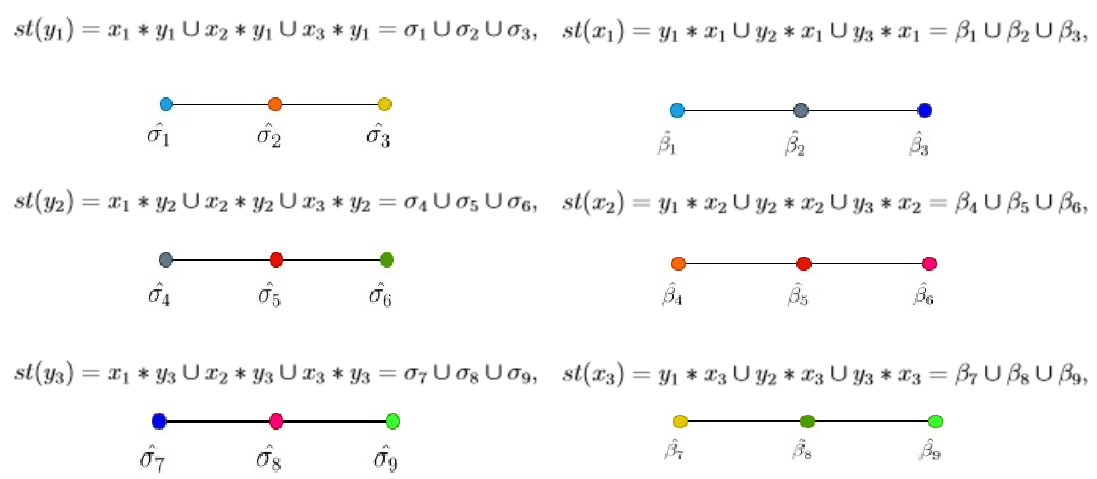}
			\label{ln1}
			\caption{The local nerves, $Nl$}
		\end{center}
	\end{figure}
	\begin{figure}[ht!]
		\begin{center}
			\includegraphics[scale=0.5]{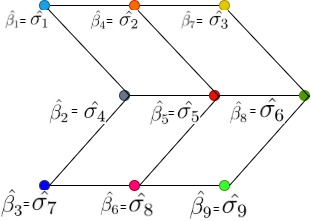}
			\caption{The global nerve, $N$}
			\label{nerve}
		\end{center}
	\end{figure}
	\begin{figure}[ht!]
		\begin{center}
			\includegraphics[scale=0.5]{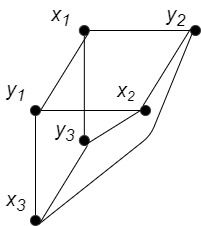}
			\caption{The complex, $K_G$.}
			\label{SGC}
		\end{center}
	\end{figure}
\end{example}

\section{Covering of game Complexes}
For a simplicial complex $K_G$, the discrete metric is defined by 
\begin{align*}
	d &: K_G\times K_G \rightarrow \{0,1\},\\\nonumber
	d(\sigma,\beta) &=
	\begin{cases}
		0, & \text{if $\sigma$ = $\beta$,}\\
		1, & o.w
	\end{cases}.
\end{align*}
The simplicial complex $K_G$ with the metric $d$ is a topological space where each simplex is open in $K_G$. 
Discrete metric on the simplicial game complex ${K}^*_G$ is defined as follows
\begin{align}\label{eqn:d2}
	d^* &:{K}_G^*\times {K}_G^*\rightarrow \{0,1\},\\\nonumber
	d^*(\sigma^*,\beta^*) &=
	\begin{cases}
		0, & \text{if $\sigma$ = $\beta$, and $f(\sigma)$=$f(\beta)$,}\\
		1, & o.w
	\end{cases}.
\end{align}
Obviously, the elements of $C^{n-1}(K_G)$ are situation simplices which are open by discrete topology induced by $d^*$.
For all $s_i^{j_i}\in S_i$, we define $$st(s_i^{j_i})=\bigcup_{X_{\hat{i}}}s_i^{j_i}\vee X_{\hat{i}}.$$ 
Clearly, there exists $\prod_{k=1,k\neq i}^nm_k$ numbers $(n-1)$-dimensional simplices in the above union. The barycentric point of each strategy in $st(s_i^{j_i})$ is
\begin{equation}\label{eqn:6}
	bc(s_i^{j_i}\vee X_{\hat{i}})=\sum_{s_{\hat{i}}\in S_{\hat{i}}}\frac{\sum_{j=1,j\neq i}^{n}x_{j_kt_j}x_{j}+x_i}{1+\sum_{j=1,j\neq i}^n\sum_{t_j=1}^{l_j} x_{j_kt_j}},
\end{equation}
where $x_i(s_i^{j_i})=1$ and $x_i(s_i^{j'_i})=0$. 
Equip each dual points in $st(s_i^{j_i})$ with a weight number by using the expected value number formula (\ref{eqn:1}). 
For each $i\in \{1,\dots,n\}$, we have $l_i$ stars of the form $st(s_i^j)$. Consider the union of these stars
\begin{align}\label{eqn:8}
	\bigcup_{s_i^{j_i}\in S_i}st(s_i^{j_i}).
\end{align}
In (\ref{eqn:8}), we have two kinds of neighborhoods. The neighbors in the $i$'th level $st(s^r_i)$, and neighbors that are defined between different levels $st_s(X_{\hat{i}})=\bigcup_{s^j_i\in S_i}s^j_i\vee X_{\hat{i}}$. In $st(s_i^r)$, every simplex which is $m$-comparable with $m\neq i$ is adjacent in the $i$'th level, and they share an $(n-2)$-simplex so we can define a dual flow between dual points of these situation facets by equation (\ref{eqn:7}), 
\begin{align}
	\omega(\hat{\sigma},\hat{\beta})=e_m(\sigma)-e_m(\beta),
\end{align}
where $\sigma =s^j_i\vee X_{\hat{i}}$ and $\beta=s^j_i\vee Y_{\hat{i}}$.
Furthermore, the neighbors $\bigcup_{s^j_i\in S_i}s^j_i\vee X_{\hat{i}}$ are $i$-comparable strategy profiles with the pure strategy for the $i$'th player. Hence, we can define dual flows in all kinds of neighborhoods.
\begin{definition}
	The number of directions enters the dual of a simplex $\sigma$ in each neighborhood 
	is called its degree on this neighborhood and is denoted by $deg(\hat{\sigma})$.
\end{definition}
 Let $Z_i$ be a nonempty subset of $S_i$ which is defined by
\begin{equation*}
	Z_i:= \{s_i^j|deg(bc(s_i^j\vee X_{\hat{i}}))=l_i-1\},
\end{equation*}
for some neighbor $st_s(X_{\hat{i}})$.
Use elements of $Z_i$ to define sub-neighbors $st_{z}(X_{\hat{i}})=\bigcup_{s_i^j\in Z_i}s_i^j\vee X_{\hat{i}}$ of $st_{s}(X_{\hat{i}})$. Now, for every $i\in \{1,\dots,n\}$ and $s_i^j\in Z_i$ define the subset $A_i(s_i^j)$ of $P_{\hat{i}}$ as a collection of all $X_{\hat{i}}$'s with $deg(bc(s_{i}^j\vee X_{\hat{i}}))=|Z_i|-1$ in the all neighbors $st_{Z_i}(X_{\hat{i}})=\bigcup_{s^{j'}_i\in Z_i}s^{j'}_i\vee X_{\hat{i}},$ which are made by $s_i^j\in Z_i$.
\begin{definition}\cite{RG}\label{def:3.2}
	Let $K$ be a simplicial complex. A pair $(C_K,p)$ is a covering complex of a $K$ if
	\itemize
	\item $C_K$ is a simplicial complex. (It is not necessarily connected.)
	\item $p:C_K\rightarrow K$ is a simplicial map, that is if $[v_0,\dots,v_k]$ is a simplex in $C_K$ then $[p(v_0),\dots,p(v_k)]$ is a simplex in $K$.
	\item For every simplex $\sigma \in K$, $p^{-1}(\sigma)$ is a union of pairwise disjoint simplices, $p^{-1}(\sigma)=\bigcup \tilde{\sigma}_i$, with $p|_{\tilde{\sigma}_i}:\tilde{\sigma}_i\rightarrow \sigma$ a bijection for every $i$.
\end{definition}
Now, we can define the covering complex for a simplicial game complex as follows.
\begin{theorem}
	Let $G$ be an $n$ player non-cooperative game. The simplicial complex
	$$C_G:=\bigvee_{j=1}^n C_j$$
	where $C_j$ is the weighted simplicial complex with simplices $((x_j,X_{\hat j}),e_j)$
	and
	the projection map $p: C_G\rightarrow K^*_G$ which is defined by 
	$$p(((x_1,X_{\hat{1}}),e_1(x_1,X_{\hat{1}}))\vee\dots\vee ((x_n,X_{\hat{n}}),e_n(x_n,X_{\hat{n}})))=(X,f(X)=\sum_{i=1}^ne_{i}(X)),$$
	is the covering complex of a simplicial game complex $K^*_G$.
\end{theorem}
\begin{proof}
	For each $X_{\hat{i}}$, construct $st_{Z_i}(X_{\hat{i}})$. Consider the dual points and corresponding weights of all simplices in $st_{Z_i}(X_{\hat{i}})$. Clearly, we can compare these weight numbers and so at least one of them is greater than others. According to the definition $A_i(s_i^j)$, the $A_i(s_i^j)$ is nonempty. Obviously, $A_i(s_i^j)\subset P_{\hat{i}}$ for every $s_i^j\in Z_i$. Since every $X_{\hat{i}}\in P_{\hat{i}}$ achieves its maximum in some $s_i^j\in Z_i$, we deduce $X_{\hat{i}}\in \bigcup_{s_i^j\in Z_i}A_i(s_i^j)$.  Hence, $\bigcup_{s_i^j\in Z_i}A_i(s_i^j)\times P_i$ is an open cover (with respect to the product topology of discrete spaces) for $P$. The cover $\bigcup_{s_i^j\in Z_i}A_i(s_i^j)\times P_i$ induces a cover on $K^*_G$ which its open sets are strategy simplices  $*\sigma_i=((x_i,X_{\hat{i}}),e_i))$, and we denote them by 
	$$C_i:=(\bigcup_{s_i^j\in Z_i}{A}_i(s_i^j)\times P_i,e_i).$$ 
	Clearly, there exist $n$ similar open covers for $i\in\{1,\dots,n\}$.
	So we set
	$$C_{G}=\bigvee_{i=1}^nC_i,$$
	which is a disjoint union equip with the product topology of discrete spaces (\ref{eqn:d2}).
	Now, it is easy to see that the projection map $p$ satisfies in Definition \ref{def:3.2}.
\end{proof}
\begin{definition}\label{def:5}
	For all $i\in N$ and for every $X_{\hat{i}}$, the Best Response $B_i(X_{\hat{i}})$ of the game complex $K^*_G$
	is the projection of an open subset $\vee_{i=1}^n*\sigma_i$ of the covering space $C_G$ with associated dual of $deg(\hat{\sigma_i})=|l_i|-1$ in the associated neighbor $(\bigcup_{X_{\hat{i}}\in A_i(s_i^j)} X_{\hat{i}})\times P_{i}$.
\end{definition}
\begin{theorem}
	The simplex $p(X)$ is a mixed Nash equilibrium simplex iff for all $i\in N$, we have $X\in B_i(X_{\hat{i}})$.
\end{theorem}
\begin{proof}
	First, we suppose $X\in B_i(X_{\hat{i}})$ for all $i\in N$. 
	Since $deg(\hat{X})=|l_i|-1$ in $(\bigcup_{X_{\hat{i}}\in A_i(s_i^j)} X_{\hat{i}})\times P_{i}$, the number $e_i(X)$ is the greatest number in $\{e_i(s_i,X_{\hat{i}})\}_{s_i\in S_i\subset P_i}$, and by using definition of degree, we get
	\begin{equation*}
		e_i(x_i,\hat{X}_i)\geq e_i(s_i,\hat{X}_i), \quad\forall s_i\in S_i,
	\end{equation*}
	and it is true for all $i\in N$. So according to the Definition \ref{Def:1}, $X$ is a Nash Equilibrium simplex. Now, if $p(X)$ is a Nash Equilibrium simplex then for every component of $X$ in the covering space, we have $e_i(x_i,\hat{X}_i)\geq e_i(s_i,\hat{X}_i)$ in $(\bigcup_{X_{\hat{i}}\in A_i(s_i^j)} X_{\hat{i}})\times P_{i}$ for all $i$. It means that 
	\begin{equation*}
		e_i(x_i,X_{\hat{i}})- e_i(s_i,X_{\hat{i}})\geq 0, \quad\forall s_i\in S_i,
	\end{equation*}	
	so $deg(\hat{X})=|l_i|-1,$ and by Def \ref{def:5}, it is a Nash Equilibrium.
\end{proof}
\section{Decomposition of a game complex}\label{DGC}
The set of dual cells of dimension $t$ is denoted by $C_t(bcsd(K_G))$ and is called the dual $m$-chain space. These spaces are used to define dual chain complex as follows
 \begin{align}\label{eqn5:1}
	 	0&\rightarrow
		C_{n-1}(bcsd(K_G)) \xrightarrow{ \hat{\partial}_{n-2}}C_{n-2}(bcsd(K_G))\xrightarrow{ \hat{\partial}_{n-1}}\dots \\\nonumber
		&\xrightarrow{ \hat{\partial}_{1}} C_{1}(bcsd(K_G))\xrightarrow{ \hat{\partial}_{0}} C_{0}(bcsd(K_G))\rightarrow 0.
\end{align}
where $\hat{\partial}$ is the boundary operator on dual cells.
Similarly, the dual $t$-cochains of the barycentric subdivision complex $bcsd({K}_G)$ is defined by
\begin{align}
	\begin{matrix}
		C^t(bcsd({K}_G),\mathbb{R}):=\{{f}_t:C_{t}(bcsd({K}_G))\rightarrow \mathbb{R}\}.
	\end{matrix}
\end{align}
Equip vector spaces of $m$-cochains $C^m(K_G,\mathbb{R})$ with an inner product by
	\begin{align}\label{eqn:5:3}
		<{f}_m,{g}_m>_m := \sum_{[{s}_0,\dots,{s}_m]\in C_{m}(bcsd({K}_G))}{f}_m(s_0,\dots,s_m){g}_m(s_0,\dots,s_m),
	\end{align}
for $0\leq m\leq n-1$. The adjoint operator of $\hat{\partial}_t$ with respect to the inner product (\ref{eqn:5:3}) is the co-boundary operator $\delta_t$. So we have the dual co-chain complex as follows
\begin{align}\label{eqn5:2}
			0\leftarrow
			&C^{n-1}(bcsd({K}_G)) \xleftarrow{ \delta_{n-1}} C^{n-2}(bcsd({K}_G))\xleftarrow{ \delta_{n-2}} \dots \\\nonumber
			&\xleftarrow{ \delta_{1}} C^{1}(bcsd({K}_G))\xleftarrow{ \delta_{0}} C^{0}(bcsd({K}_G))\leftarrow 0.
\end{align}
The $t$-Laplace Beltrami operator on a dual $t$-cochain space $C^t(bcsd(K_G),\mathbb{R})$
is defined by 
\begin{equation*}
	\Delta_t=\delta_{t-1}o\hat{\partial}_{t-1}+\hat{\partial}_{t}o\delta_{t}.
\end{equation*}
We say that a dual $t$-cochain $f_t$ is harmonic if $\Delta_t(f_t)=0.$ 
\begin{definition}\label{Def5:1}
	A potential game $P_G$ is a subspace of $C^{n-1}(K_G)$ for which there exists a function $\Phi:P\rightarrow \mathbb{R}$ satisfying 
	\begin{align*}
		\Phi(x_i,X_{\hat{i}})-\Phi(y_i,Y_{\hat{i}}) &= e_i(x_i,X_{\hat{i}})-e_i(y_i,Y_{\hat{i}}),\\
		&= \omega(\hat{X},\hat{Y}),
	\end{align*}
in every component of the covering space $K^*_G$.\\
The sub-space $H_G$ of $C^{n-1}(K_G)$ is called a harmonic game for which the associated dual flow subspace of $p^{-1}(*\sigma)$ for every $*\sigma$ in $K^*_G$  is satisfied $\Delta_1w=0$. That is
\begin{equation*}
	(\Delta_{1}w)(\hat{X},\hat{Y})=\sum_{i=1}^n\Delta_{1}(e_i(x_i,X_{\hat{i}})-e_i(y_i,Y_{\hat{i}}))=0.
\end{equation*}
\end{definition}
\begin{theorem}
	The vector space $C^{n-1}(K^*_G)$ has an orthogonal decomposition of the potential, harmonic and non-strategic
	games, i.e.
	\begin{equation*}
		C^{n-1}(K^*_G)=P_G\oplus H_G\oplus N_G.
	\end{equation*}
\end{theorem}
\begin{proof}
	From the chain complex (\ref{eqn5:1}) and (\ref{eqn5:2}), we have $C^{1}(bcsd(K_G))=ker(\delta_{1})\oplus im(\hat{\partial}_{1})$. Furthermore, $\delta_1$ is a linear operator on  a vector space $C^1(bcsd(K_G))$ and $ker(\delta_1)$ is the subspace of $C^1(bcsd(K_G))$, so we have $C^{1}(bcsd(K_G))=ker(\delta_1)\oplus (ker(\delta_1))^\bot$.
	Since we have $(ker(\delta_1))^\bot=im(\hat{\partial}_1)$, we get $C^{1}(bcsd(K_G))=ker(\delta_1)\oplus im(\hat{\partial}_1)$.
	On the other hand, we have
	\begin{align*}
		ker(\delta_{1})&=C^1(bcsd(K_G))\bigcap ker(\delta),\\
		&=[im(\delta_{0})\oplus ker(\hat{\partial}_{0})]\bigcap ker(\delta_{1}),\\
		&=[ker(\delta_{1})\bigcap ker(\hat{\partial}_{0})]\oplus[ker(\delta_{1})\bigcap im(\delta_{0})].
	\end{align*}
Furthermore, $[ker(\delta_{1})\bigcap ker(\hat{\partial}_{0})]=ker(\hat{\partial}_1o\delta_1+\delta_0o\hat{\partial}_0)$ and $ker(\delta_{1})\bigcap im(\delta_{0})=im(\delta_0)$, so we get a decomposition for $C^1(bcsd(K^*_G))$ as follows
\begin{align*}
C^1(bcsd(K^*_G))=im\hat{\partial}_1\oplus ker(\hat{\partial}_1o\delta_1+\delta_0o\hat{\partial}_0)\oplus im(\delta_0).
\end{align*} 
On the other hand, each flow in these components determine some adjacent situations in the game complex $K^*_G$. So according to the Definition \ref{Def5:1}, the adjacent situations of the flows in $im\hat{\partial}_1$ is denoted by $P_G$ and the adjacent situations of the flows in $ker(\hat{\partial}_1o\delta_1+\delta_0o\hat{\partial}_0)$ and $im(\delta_0)$ are denoted by $H_G$ and $N_G$, respectively. Obviously, we get $C^{n-1}(K^*_G)=P_G\oplus H_G\oplus N_G$.
\end{proof}


\section*{Acknowledgment}
This research was conducted at Amirkabir university (Polytechnic of Tehran) with funding by the Science Research Institute.


\begin{thebibliography}{99}
	
\bibitem{AH} 
Allen Hatcher. Algebraic topology. Cambridge University Press, Cambridge, 2002.

\bibitem{ANH} 
Anil N. Hirani. Discrete exterior calculus, PhD thesis, 2003.

\bibitem{CX} 
Chen Xilin, Topology in economical analysis, $https://rpubs.com/fduzhin/xilin_fyp_report$, 2016.

\bibitem{JVN}
John von Neumann and Oskar Morgenstern, Theory of games and economic behavior, Princeton University Press, 1994.

\bibitem{MDS} 
Mark Meyer, Mathieu Desbrun, Peter Schruoder, and Alan H. Barr. Discrete differential-geometry Operatorsfor triangulated 2-manifolds. In International Workshop on Visualization and Mathematics, VisMath, 2002.

\bibitem{MDS}
Maurice Herlihy, Dmitry Kozlov, Sergio Rajsbaum,
Distributed Computing Through Combinatorial Topology,  Elsevier, 2014.

\bibitem{CMOP} 
Ozan Candogan, Ishahi Menache, Asuman Ozdaglar and Pablo A. Parribo, Flows and decomposition of Games: Harmonic and Potential Games, Math. Oper. Res., 36, no. 3, pp. 474-535, 2011.

\bibitem{RG}
Richard Gustavson, Laplacian of covering complexes, Rose-Hulman Undergraduate Mathematics Journal, Vol 12, 2011.

\bibitem{SE}
Sarah Egan, Nash equilibria in games and simplicial complexes, P.h.D thesis Bath university, 2008.
 
 \bibitem{JoelWatson}
 Watson, Joel, “Strategy: an introduction to game theory, Third Edition.”, WW Norton; 2013.
 
 \bibitem{JL}
 Xiaoye Jiang, Lek-Heng Lim, Statistical ranking and combinatorial Hodge theory, Mathematical Programming , Springer, 2011.
 
 
\end{thebibliography}
	
{Amirkabir University of Technology (Polytechnic of Tehran),\\
Department of  Mathematics and Computer science,\\ 
	shojaee$_{-}$neda@aut.ac.ir\\
	mmreza@aut.ac.ir}
\end{document}